\documentclass[11pt]{amsart}
\usepackage{amssymb, amsmath, amsthm}
\usepackage{color}
\usepackage{graphicx}
\usepackage[final]{hyperref}
\usepackage[a4paper, centering]{geometry}
\usepackage{pdfsync}

\usepackage[mathscr]{eucal}
\usepackage{lmodern}

\newif\ifextended
\extendedfalse

\newtheorem{teor}{Theorem}[section]

\theoremstyle{definition}
\newtheorem{Definition}[teor]{Definition}
\theoremstyle{remark}
\newtheorem{rk}[teor]{Remark}

\long\def\elimina#1{}

\def\R{\mathbb{R}}
\def\C{\mathbb{C}}
\def\N{\mathbb{N}}

\def\M{\mathcal{M}_b(\Omega;\R^{3\times 3})}
\def\Rtt{\R^{3\times 3}}
\def\Ltm{L^{3/2}(\Omega,\Rtt)}

\def\bmu{\beta^\mu}
\def\e{\varepsilon}
\def\ue{u_{h}}
\def\ve{v_{h}}
\def\Ce{\C_{\varepsilon}}
\def\Chom{\widehat{\C}}
\def\tens{a_{ij}^{hk}}
\def\ind{\{1,2,3\}}
\def\E{\mathcal{E}}
\def\S{\mathcal{S}}
\def\A{\mathcal{A}}
\def\Rort{\mathcal{R}^\perp}
\def\X{\mathcal{X}}
\def\Mant{\overline{H}^a}

 \DeclareMathOperator{\dive}{div}
 
 \DeclareMathOperator{\curl}{curl}

\begin{document}
\title[Linear elasticity problems with incompatible deformation 
fields]%
{Duality arguments for linear elasticity problems with incompatible deformation 
fields} 
\author{{Adriana Garroni}\and {Annalisa Malusa}} 
\address{Dipartimento di Matematica ``G.\ Castelnuovo'', Sapienza Università di 
Roma\\ P.le Aldo Moro 2 -- 00185 Roma (Italy)}
\email{{garroni@mat.uniroma1.it} \and {malusa@mat.uniroma1.it}}
\date{May 1, 2020}
\date{\today}

\begin{abstract}
We prove existence and uniqueness for solutions to equilibrium problems for 
free--standing, traction--free, non homogeneous crystals in the presence of 
plastic slips. 
Moreover we prove that this class of problems is closed under $G$-convergence 
of the operators. In particular the homogenization procedure, valid for 
elliptic systems in linear elasticity, depicts the macroscopic features of 
a composite material in the presence of plastic deformation.
\end{abstract}

\maketitle

\begin{center}
\textit{To Umberto Mosco, with gratitude and appreciation  \\
for having us
revealed 
first the beauty of the Calculus of Variation}
\end{center}

\section{Introduction}
In this paper we consider linear boundary value problems for systems of the form
\begin{equation}\label{f:intrpb}
\begin{cases}
-\dive (\C(x) \beta)=0, & \text{in}\ \Omega, \\
\curl \beta =\mu, & \text{in}\ \Omega, \\
\C(x)\beta \cdot n =0, & \text{on}\ \partial\Omega, \\
\end{cases}
\end{equation}
where  $\Omega$ is a bounded, smooth, open subset of $\R^3$, $\C(x)$ is a 
symmetric tensor--valued function in $\Omega$ with VMO coefficients, satisfying 
the standard 
hypotheses of elasticity theory (see Definition \ref{d:ellip}), and $\mu$ is a 
matrix--valued bounded Radon measure in $\Omega$.

In  linear elasto-plastic models $\C(x)$ is the elastic tensor for
a non homogeneous elastic body and the gradient of the displacement field is decomposed in plastic and elastic strain. The field $\beta \in L^1(\Omega; \Rtt)$ in \eqref{f:intrpb} represents the elastic strain which, in the presence of a non trivial plastic deformation (possibly due to non homogenous plastic slips), may be non compatible, i.e., it may not be curl free. 
The incompatibility is the effective result of a distribution of crystals defects (the dislocations) and the measure $\mu$ represents the dislocation density.
Dislocations are topological defects in the crystalline structure that,
at a mesoscopic level, can be identified with loops along which the strain has 
a singularity. In particular the strain field is curl free outside the loops 
and has a non trivial circulation around the lines. Therefore these 
singularities are nicely described by measures supported on $1$--rectifiable 
closed curves with  matrix valued multiplicities that depend on the underlined 
crystalline structure and the orientation of the line, i.e., $\curl\beta = 
b\otimes \tau \mathcal{H}^1 |\gamma$ (see e.g.
\cite{CGO}, Section 2.2., and the references therein, for a detailed 
description of the kinematics of plastic deformations). The topological nature 
of these defects is transparent in the fact that the line $\gamma$ is a 
collection of closed curves with constant multiplicity $b$ (the Burgers vector) 
and it is translated  in the constraint $\dive\mu=0$. The collective effect of 
dislocations produces an effective strain field $\beta$ whose curl is given by 
an arbitrary Radon measure as in \eqref{f:intrpb}. 

We then prove that for every $\mu$ such that $\dive \mu=0$, there exists a unique 
distributional solution 
$\beta \in 
L^{3/2}(\Omega,\R^{3\times 3})$ to the system \eqref{f:intrpb}. 

Following \cite{CGO} the result is obtained by decoupling the problem and finding the solution in 
the form $\beta=\bmu+Du$, where $\curl \bmu=\mu$ and $u$ is the solution of the 
elliptic system
\begin{equation}\label{f:ellintr}
\begin{cases}
-\dive (\C D u)=\dive (\C \beta^\mu), \quad & \text{in}\ \Omega, \\
\C Du \cdot n =-\C\beta^\mu \cdot n, & \text{on}\ \partial\Omega.
\end{cases}
\end{equation}

The existence of $\bmu \in L^{3/2}(\Omega,\R^{3\times 3})$ is guaranteed by 
the celebrated result by Bourgain and Brezis \cite{BrBu}, while the existence 
and uniqueness (up to rigid infinitesimal rotations, i.e., antisymmetric matrices) of a distributional solution $u \in W^{1,3/2}(\Omega;\R^3)$ to
the non variational problem \eqref{f:ellintr} is obtained by adapting the 
method of duality solutions for elliptic equations with measure forcing terms 
(see \cite{St}, \cite{DMOP}, and, e.g., \cite{Bo1}, \cite{Bo2}, \cite{MO}, 
\cite{OP}, \cite{Shen}).

In the second part of the paper we deal with the asymptotic behavior, as $h \to 
0$, of the solution $\beta_h$ of the problems
\begin{equation}\label{f:intrh}
\begin{cases}
-\dive (\C_h(x) \beta_h)=0, & \text{in}\ \Omega, \\
\curl \beta_h =\mu, & \text{in}\ \Omega, \\
\C_h(x)\beta_h \cdot n =0, & \text{on}\ \partial\Omega. \\
\end{cases}
\end{equation}
Assuming that the linear elliptic operators associated to the coefficients 
$\C_h$ $G$-converge to the operator associated to $\C_0$, under suitable 
uniform 
conditions for $\C_h$, we prove the weak  convergence in  
$L^{3/2}(\Omega,\R^{3\times 3})$  of $\beta_h$ to the unique (up to rigid infinitesimal rotations) solution $\beta_0$
to the problem
\begin{equation}\label{f:intrlim}
\begin{cases}
-\dive (\C_0(x) \beta_0)=0, & \text{in}\ \Omega, \\
\curl \beta_0 =\mu, & \text{in}\ \Omega, \\
\C_0(x)\beta_0 \cdot n =0, & \text{on}\ \partial\Omega. \\
\end{cases}
\end{equation}
This result is also obtained by decoupling 
the boundary value problems and by investigating the asymptotic behaviour of the 
duality solutions of elliptic systems with non regular forcing terms under the 
assumption of $G$-convergence of the differential operators.
We conclude by discussing the special case of the homogenization.

Finally we remark that the duality arguments and therefore the regularity properties assumed for the elastic tensor $\C$ are needed in order to deal with the cases in which the curl of the strain $\beta$ is assigned to be singular. This is the case when the plastic strain is concentrated on low dimensional sets in $\R^3$. In particular in the presence of single dislocations, when the measure $\mu$ is concentrated on 1-rectifiable lines, the field $\beta$ is not in $L^2$. On the other hand if $\mu\in H^{-1}(\Omega,\Rtt)$ the duality argument is not necessary, the problem is variational and it can be studied minimizing the corresponding elastic energy and the asymptotics can be obtained via $\Gamma$-convergence.

\section{Notations and basic hypotheses on the operators}\label{s:not}
In what follows $\Omega \subseteq \R^3$ will be a open, bounded, simply 
connected set with $C^1$ boundary, and $\M$ will denote the set 
of all bounded matrix--valued Radon measures on $\Omega$.

The subspaces of $\Rtt$ of all symmetric matrices and of all skew--symmetric 
matrices will be denoted by $\S$ and $\A$, respectively.

If $\C=(\tens)$, $i,j,h,k\in\ind$, is a fourth--order tensor, and 
$\xi=(\xi_{ij})$, $\eta=(\eta_{ij})$ $i,j\in \ind$
are square matrices of order $3$, we set  
\[
\C\xi =\left(\sum_{h,k=1}^3\tens \xi_{hk}\right)_{i,j\in \ind} \quad 
\C\xi \cdot \eta  = \sum_{i, j,h,k=1}^3\tens \xi_{hk}\eta_{ij} \qquad 
\|\xi\|= \left(\sum_{i,j=1}^3|\xi_{ij}|^2 \right)^{\frac{1}{2}}
\]

\begin{Definition}\label{d:ellip}
Given $c_0,c_1>0$, we denote by $\E(c_0,c_1,\Omega)$ the set of all 
tensor--valued functions $\C(x)=(\tens(x))$, $x\in\Omega$,  such that the 
following 
hold:
\begin{enumerate}
\item $\tens\in L^\infty(\Omega)$ for all $i,j,h,k\in\ind$;
\item $a_{ij}^{hk}=a_{ji}^{hk}=a_{hk}^{ij}$ for all $i,j,h,k\in\ind$;
\item $c_0 \|\xi+\xi^T\|^2\leq \C(x) \xi \cdot \xi \leq c_1 \|\xi+\xi^T\|^2$ for a.e.\ 
$x\in\Omega$, and 
for every $\xi\in \Rtt$.
\end{enumerate}

\end{Definition}

\begin{rk}
The symmetry assumption in Definition \ref{d:ellip} implies that $\C\xi=0$
for every $\xi \in \A$. 
On the other hand, if $\C\in \E(c_0,c_1,\Omega)$, $\C(x) \xi \cdot \xi$ is a positive 
definite continuous quadratic form on $\S$.  
\end{rk}

For a matrix valued distribution $V\colon \Omega \to \Rtt$, $\dive V$ and 
$\curl V$ denote the row--wise 
distributional divergence and curl of $V$ respectively.
In particular, given a tensor--valued function $\C\colon \Omega \to 
\R^{(3\times3)^2}$ and
a matrix--valued function $\beta\colon \Omega \to \Rtt$, the vector $\dive 
(A(x)\beta(x))\colon \Omega \to \R^3$ has components
\[
(\dive (\C\beta))_i=
\sum_{j=1}^3 \frac{\partial}{\partial x_j}\left(\sum_{h,k=1}^3 \tens \beta_{hk} 
\right),
\qquad i\in\ind.
\]

\section{Preliminary results on elliptic systems}\label{s:known}
In this section we recall the basic existence and regularity
results concerning boundary value problems for elliptic systems of PDEs
of the form
\begin{equation}\label{f:elliptic}
\begin{cases}
-\dive (\C(x) D v)=\dive G, & \text{in}\ \Omega, \\
\C Dv \cdot n =-G\cdot n, & \text{on}\ \partial\Omega. \\
\end{cases}
\end{equation}
For $G\in L^2(\Omega, \Rtt)$  we deal with 
variational solutions in $W^{1,2}(\Omega,\R^3)$. Precisely a function  $v$ is a weak 
solution of \eqref{f:elliptic} if $v\in W^{1,2}(\Omega,\R^3)$, and
\[
\int_{\Omega}\C Dv \cdot D\varphi\, dx = 
-\int_{\Omega} G \cdot D\varphi\, dx, \qquad 
\forall \varphi\in W^{1,2}(\Omega;\R^3).
\]
Choosing as test function the rigid movement  $\varphi(x)=\xi x+b$, with   
$b\in\R^n$ and $\xi\in\A$, we obtain
\[
\int_{\Omega}\C Dv \cdot \xi\, dx = 
-\int_{\Omega} G \cdot\xi\, dx.
\]  
If we assume that the coefficients $\C$ belong to the class $\E(c_0,c_1,\Omega)$, 
we have that
\[
\int_{\Omega}\C Dv \cdot \xi\, dx = \int_{\Omega}\C \xi\cdot Dv\, dx =0
\] 
so that the existence of a weak solution $v$ to 
\eqref{f:elliptic} implies that
\begin{equation}\label{f:compcond}
\int_{\Omega}G \, dx\in \S.
\end{equation}

The compatibility condition \eqref{f:compcond} on the forcing term $G$
must be required and the solutions of \eqref{f:elliptic} will be defined up to
additive  rigid transformations belonging to the set
\[
\mathcal{R}=\{\varphi(x)=\xi x+b,\quad  b\in\R^n,\ \xi\in\A\}.
\]

In what follows  $\X_p$, $p>1$, will denote the set of admissible forcing terms 
in $L^p$
\[
\X_p=\{G\in L^p(\Omega,\Rtt) \ \text{such that \eqref{f:compcond} holds true}\},
\]
and,  with a little abuse of notation, $\Rort$ will denote the following set
\begin{equation}\label{p-orth}
\Rort:=\{u \in W^{1,1}(\Omega,\R^3): \ \int_\Omega u \,dx =0,\  \int_\Omega Du\,dx \in \S\}.
\end{equation}
 
The  existence result below for \eqref{f:elliptic}  is based on the second Korn inequality (see, \cite{OSY}, Theorem 
2.5)
\begin{equation}\label{f:korn}
\int_{\Omega} |u|^2\,  dx + \int_{\Omega} |Du|^2\,  dx \leq C 
\int_{\Omega}|Du+(Du)^T|^2\, dx,
\qquad \forall u \in W^{1,2}(\Omega,\R^3)\cap \Rort,
\end{equation}
and Lax--Milgram Theorem (see e.g.\cite{Shen}, Theorem 1.4.4).

\begin{teor}
Suppose that $\Omega$ is a bounded Lipschitz domain in $\R^3$, and $\C\in 
\E(c_0,c_1,\Omega)$. 
Then for every $G\in \X_2$ there exists a 
weak solution $v\in 
H^1(\Omega,\R^3)$ to problem \eqref{f:elliptic},
unique up to rigid displacements in $\mathcal{R}$, i.e. unique in $\Rort$.
\end{teor}

In what follows we will need a  $W^{1,p}$ estimate for the weak solution to
\eqref{f:elliptic} with forcing term in  $L^p$, $p>2$. The higher summability 
of the solution, valid for operators with constant coefficients, fails to be 
true for general elliptic systems (see, e.g., \cite{ADN}).

Hence, from now on, we assume in addition that the coefficients belong to 
$VMO$, that is, setting
\[
\omega_\Omega(\C,r):= \sup_{B_\rho\subseteq \Omega,\ \rho\leq r} 
\frac{1}{|B_\rho|}\int_{B_\rho} \left| \C(s)- \frac{1}{|B_\rho|}
\int_{B_\rho} \C(t)\, dt \right|\, ds
\] 
we assume that 
\begin{equation}\label{f:VMO}
\lim_{r \to 0^+}\omega_\Omega(\C,r)=0.
\end{equation}

\begin{teor}\label{t:reg}
Assume that $\Omega$ is a bounded $C^1$ domain in $\R^3$, $p\geq 2$, 
and $\C\in 
\E(c_0,c_1, \Omega)$ such that \eqref{f:VMO} holds. 
Then for every $G\in \X_p$ there exists a unique weak solution $v$ to 
\eqref{f:elliptic} 
in  $W^{1,p}(\Omega;\R^3)\cap \Rort$  which  satisfies
\[
\int_{\Omega}\C Dv \cdot D\varphi\, dx = 
-\int_{\Omega} G \cdot D\varphi\, dx, \qquad 
\forall \varphi\in W^{1,p'}(\Omega;\R^3).
\]
Moreover, there exists 
a constant $C>0$, depending only on $p, c_0,c_1$ $\Omega$, and 
$\omega_\Omega(\C,r)$,  such that 
\begin{equation}\label{f:regest}
\|v\|_{W^{1,p}(\Omega,\R^3)} \leq C \|G\|_{L^P(\Rtt)}.
\end{equation}
\end{teor}
\begin{proof}
See \cite{Shen}, Theorem 5.6.4.
\end{proof}

\begin{rk}
In what follows, we will consider as the unique  weak solution to 
\eqref{f:elliptic} the one  orthogonal to the set of rigid transformations 
$\mathcal{R}$.
\end{rk}

\section{Existence}\label{s:exist}

This section is devoted to the proof of the following result.

\begin{teor}\label{t:ex}
Suppose that 
\begin{enumerate}
\item $\C\in \E(c_0,c_1,\Omega)$ satisfying \eqref{f:VMO};
\item $\mu \in \M$  with  $\dive \mu=0$.
\end{enumerate}  
Then there is a distributional solution $\beta \in 
L^{3/2}(\Omega,\R^{3\times 3})$ to the system
\begin{equation}\label{f:mainpb}
\begin{cases}
-\dive (\C(x) \beta)=0, & \text{in}\ \Omega, \\
\curl \beta =\mu, & \text{in}\ \Omega, \\
\C\beta \cdot n =0, & \text{on}\ \partial\Omega. \\
\end{cases}
\end{equation}
The solution is unique (up to an additive constant antisymmetric matrix), and 
there exists a constant $c>0$, depending only on $\Omega$, 
$\omega_\Omega(\C,r)$, and 
$c_0, c_1$,  such that
\begin{equation}\label{f:stimab}
\|\beta-\bar\beta^a\|_{\Ltm}\leq c |\mu|(\Omega),
\end{equation}
where $\bar\beta^a$ denotes the average of the antisymmetric part of $\beta$, i.e., $\bar\beta^a=\frac{1}{2|\Omega|}\int_\Omega (\beta-\beta^T)dx$.
\end{teor}

The proof is based on a suitable decomposition $\beta=\bmu+Du$, with $\bmu$
such that  $\curl \bmu =\mu$, and 
$u$ weak solution of an elliptic problem.
The uniqueness then follows by the linearity of the problem.

Concerning the purely incompatible part of $\beta$, we use the following 
well--known result by Bourgain and Brezis 

\begin{teor}\label{p:betamu}
For every $\mu \in \M$ with 
$\dive 
\mu=0$ there exists a field $\beta^\mu \in \Ltm$ such that
\begin{enumerate}
\item $\curl \beta^\mu=\mu$,
\item $\|\beta^\mu\|_{\Ltm}\leq c |\mu|(\Omega)$.
\end{enumerate}
\end{teor}
\begin{proof}
See \cite{BrBu}, \cite{CGO}.
\end{proof}

Concerning the potential part $Du$ of $\beta$, we adapt to the problems of 
linear 
elasticity the 
method of duality solutions for elliptic equations with measure forcing terms 
(see \cite{St}, \cite{CGO}, \cite{Shen}), in order to obtain a selected 
distributional 
solution 
$u$ to  the non variational elliptic problem
\[
\begin{cases}
-\dive (\C(x) D u)=\dive (\C \bmu), & \text{in}\ \Omega, \\
\C Du \cdot n =-\C\bmu \cdot n, & \text{on}\ \partial\Omega, \\
\end{cases}
\]
with forcing term given by a field belonging to $L^{\frac32}(\Omega;R^3)$.

The starting point for the formulation by duality of elliptic problems is the 
following. Let $F$ and $G\in \X_2$, and let $w$, $v$ be the weak 
solutions to \eqref{f:elliptic} with datum $F$ and $G$, respectively. Choosing 
$w$ as test function in the equation solved by $v$ and conversely, thanks to 
the symmetry of the tensor $\C$ we obtain
\begin{equation}\label{f:dual2}
\int_{\Omega}G \cdot Dw \, dx = \int_{\Omega}F \cdot Dv \, dx.
\end{equation}
If, in addition, $G\in L^3(\Omega,\Rtt)$, then, by Theorem \ref{t:reg},
$v\in W^{1,3}(\Omega,\R^3)$, so that \eqref{f:dual2} is well defined when 
the forcing term $F$ belongs to $L^{3/2}(\Omega,\Rtt)$ and the corresponding
``solution'' $w$ belongs to $W^{1,3/2}(\Omega,\R^3)$. This fact inspires the 
following definition of weak solution for \eqref{f:elliptic} when the forcing 
term
is in $L^{3/2}(\Omega,\Rtt)$.

\begin{Definition}
Let $\C\in \E(c_0,c_1,\Omega)$, and  $F\in \X_{3/2}$.
A function $u$ is a duality solution to the elliptic problem
\[ 
\begin{cases}
-\dive (\C(x) D u)=\dive F, & \text{in}\ \Omega, \\
\C Dv \cdot n =-F\cdot n, & \text{on}\ \partial\Omega, \\
\end{cases} 
\]
if $u \in W^{1,3/2}(\Omega,\R^3)$, and
 
\begin{equation*}
\int_{\Omega}G \cdot Du \, dx = \int_{\Omega}F \cdot  Dv \, dx
\end{equation*}
for every $G\in \X_3$, where $v$ is the weak solution to 
\eqref{f:elliptic}. 
\end{Definition}

\begin{rk}\label{r:sub}
Given $u \in W^{1,3/2}(\Omega,\R^3)\cap\Rort$, and $H\in L^3(\Omega,\R^3)$,
and setting
\[
G=H-\Mant\qquad\hbox{with}\qquad \Mant= \frac{1}{|\Omega|}\int_{\Omega}\frac{H-H^T}{2},
\]
we have that $G\in \X_3$ and 
\[
\int_{\Omega}H Du\, dx= \int_{\Omega}G Du\, dx + \Mant \int_{\Omega}Du\, dx =
\int_{\Omega}G Du\, dx
\]
where in the last equality we have used the fact that $\Mant\in \A$, and 
$\int_{\Omega}Du\, dx\in \S$.
\end{rk}

The next result shows that the duality solution exists, is unique (up to rigid transformations in $\mathcal{R}$), and it is 
the unique solution in the sense of distributions which can be obtained as limit 
of variational solutions of the same problem. 

\begin{teor}\label{p:dual}
Let $\C \in \E(c_0,c_1, \Omega)$ satisfying \eqref{f:VMO}, and $F\in \X_{3/2}$.
Then there exists a unique function $u\in W^{1,3/2}(\Omega;\R^3)\cap \Rort$, such that
the following holds:
\begin{enumerate}
\item  $u$ is a duality solution to
\begin{equation}\label{f:distr}
\begin{cases}
-\dive (\C(x) D u)=\dive F, & \text{in}\ \Omega, \\
\C Du \cdot n =-F \cdot n, & \text{on}\ \partial\Omega;
\end{cases}
\end{equation}
\item $u$ is a solution obtained by approximation: for every sequence 
$(F_k)\subseteq \X_3$ converging to $F$
in $\Ltm$, the sequence $v_k$ of solutions in $W^{1,2}(\Omega, \R^3)\cap \Rort$,  to the problems
\begin{equation}\label{f:elliptick}
\begin{cases}
-\dive (\C(x) D v_k)=\dive F_k, & \text{in}\ \Omega, \\
\C Dv_k \cdot n =-F_k \cdot n, & \text{on}\ \partial\Omega, \\
\end{cases}
\end{equation} 

converges to $u$ in the weak topology of $W^{1,3/2}(\Omega;\R^3)$.

\item $u$ is a distributional solution to \eqref{f:distr},
and  there exists a constant $c>0$, depending only on $c_0, c_1$, $\Omega$ and
$\omega_\Omega(\C,r)$, such that
\begin{equation}\label{f:estw}
\| u\|_{W^{1,3/2}(\Omega;\R^3)} \leq c \|F\|_{\Ltm}.
\end{equation}
\end{enumerate}

\end{teor}

\begin{proof}
For any given $G\in \X_3$, let $v\in W^{1,3}(\Omega; \R^3)$ be the 
solution to \eqref{f:elliptic} with right hand side $\dive G$. 
Let $v_k$ be the
solutions  to \eqref{f:elliptick}.

Since  both $v$ and $v_k$ are variational solutions, by the symmetry assumption 
on $\C$, we obtain that the equality
\begin{equation}\label{f:duak}
\int_{\Omega}F_k \cdot Dv \, dx =
\int_{\Omega}\C Dv_k \cdot Dv \, dx =\int_{\Omega}\C Dv \cdot Dv_k \, dx
= \int_{\Omega}G \cdot Dv_k \, dx
\end{equation}
holds for every $k\in\N$, and hence 
we get the estimate
\begin{equation}\label{f:duak2}
\left|\int_{\Omega}G \cdot Dv_k \, dx \right|\leq 
\|F_k\|_{\Ltm} \|v\|_{W^{1,3}(\Omega;R^3)} \leq M 
\|G\|_{L^{3}(\Omega;R^3)}
\end{equation}
for every $G\in \X_3$.
Using the fact that the sequence $F_k$ is equibounded in $\Ltm$, the 
regularity estimate \eqref{f:regest} for $v$, and Remark \ref{r:sub},  we 
conclude that
\[
\|Dv_k\|_{{L^{3/2}(\Omega;R^3)}}=
\sup_{G\in L^{3}(\Omega;R^3)} \frac{1}{\|G\|_{L^{3}(\Omega;R^3)}}
\left| 
\int_{\Omega}G \cdot Dv_k \, dx \right|
\leq M,
\] 
so that there exists a subsequence (still denoted by $v_k$) converging to a 
function $u$ in the weak topology of $W^{1,3/2}(\Omega;\R^3)$.

Since $u$ is the weak limit of distributional solutions, it follows that 
it is also a distributional solution, while the fact that $u$ is a duality 
solution follows passing to 
the limit in \eqref{f:duak} as $k 
\to \infty$. Moreover, since $W^{1,3/2}(\Omega;\R^3)\cap\Rort$ is a weakly closed subspace of  
$W^{1,3/2}(\Omega;\R^3)$, we also obtain that $u\in\Rort$.

The estimate \eqref{f:estw} follows from the very definition of duality 
solution. Specifically  as above we have
\[
\left|\int_{\Omega}G \cdot Du \, dx \right| \leq  \|F\|_{\Ltm} 
\|v\|_{W^{1,3}(\Omega;R^3)} \leq c \|F\|_{\Ltm} 
\|G\|_{L^{3}(\Omega;R^3)}
\]
for every $G\in \X_{3}$, and hence, by Remark \ref{r:sub},
for every $G\in L^{3}(\Omega;R^3)$, which gives  \eqref{f:estw}.

It remains to show that the duality solution is unique in $W^{1,3/2}(\Omega;\R^3)\cap \Rort$. 
Suppose that both
$u$ and $w$ belong to  $\Rort$ and satisfy \eqref{f:dual2}.
Then we have
\[
\int_{\Omega}G [Du-Dv]\, dx = 0, \qquad \forall G\in \X_3,
\]
and, by Remark \ref{r:sub},
\[
\int_{\Omega}G [Du-Dv]\, dx = 0, \qquad \forall G\in L^{3}(\Omega;R^3).
\]
This implies that $u=v$ in $\Omega$.
In particular, we recover that the whole sequence $v_k$ weakly converges to the 
solution $u$ of \eqref{f:dual2}. 

Finally, if $(\widetilde{F}_k)\subseteq L^3(\Omega; \Rtt)$ is another sequence 
converging to $F$ in $\Ltm$, the previous arguments show that 
the sequence $(\widetilde{v}_k)$ of solutions of the problems with forcing 
term  $\widetilde{F}_k$ converge to the duality solution $u$ of 
\eqref{f:distr}, which turns out to be the unique 
distributional solution of problem \eqref{f:distr} that can be obtained by 
approximation. 
\end{proof}

As a consequence of Propositions \ref{p:betamu} and \ref{p:dual}, we obtain the 
existence result in elasticity stated in Theorem \ref{t:ex}

\begin{proof}[Proof of Theorem \ref{t:ex}]
Let  $\beta^\mu \in \Ltm$ be as in Theorem \ref{p:betamu}. 
Since $F=\C\bmu$ belongs to $\Ltm$ and satisfies the compatibility condition 
\eqref{f:compcond}, by Theorem \ref{p:dual} there exists $u$ duality solution
to the problem
\[
\begin{cases}
-\dive (\C(x) D u)=\dive (\C \beta^\mu), & \text{in}\ \Omega, \\
\C Du \cdot n =-\C\beta^\mu \cdot n, & \text{on}\ \partial\Omega. 
\end{cases}
\]
The field $\beta=\bmu+Du$ provides a distributional solution to 
\eqref{f:mainpb}. 
The uniqueness (up to constant antisymmetric matrices) then follows from the 
linearity of the problem. 

Finally, by Theorem \ref{p:betamu}(ii), \eqref{f:estw}, and the boundedness 
of the coefficients $\C$, we get
\[
\|\beta\|_{\Ltm} \leq \|\bmu\|_{\Ltm}+ \|Du\|_{\Ltm}\leq C |\mu|(\Omega).
\]
In particular, $\beta \in L^1(\Omega, \Rtt)$, and
\[
|{\overline{\beta}}^a|=
\left|\frac{1}{2|\Omega|}\int_\Omega (\beta-\beta^T)\, dx\right| \leq C
\|\beta\|_{L^1(\Omega, \Rtt)} \leq C |\mu|(\Omega),
\]
and estimate \eqref{f:stimab} follows. 
\end{proof}

\begin{rk}
The unique solution $\beta$ to problem \eqref{f:mainpb} admits infinitely many
representation of the form $\beta=\bmu+Du$, depending on the choice of $\bmu$.
\end{rk}

\section{$G$-convergence for problems in elasticity}\label{s:gconv}

We are now interested on the behaviour of the solutions to \eqref{f:mainpb}
corresponding to varying operators. In the framework of elliptic systems, the 
convergence of solutions is encoded in the notion of $G$-convergence (see, 
e.g. \cite{JKO}, \cite{DGS}, \cite{Sp}, \cite{MT}).  

\begin{Definition}
A sequence of tensor--valued functions $(\C_h)\in \E(c_0,c_1,\Omega)$ is said to be
G--convergent to  $\C_0\in \E(c_0,c_1,\Omega)$ if for any $f\in W^{-1,2}(\Omega,\R^3)$ 
the solution
$v_h \in W^{1,2}_0(\Omega,\R^3)$ to the Dirichlet problems
\[
\begin{cases}
-\dive (\C_h(x) D v_h)=f, & \text{in}\ \Omega, \\
v_h=0, & \text{on}\ \partial\Omega
\end{cases}
\]
converge in weak topology of  $W^{1,2}_0(\Omega,\R^3)$, as $h\to 0$, to the 
solution $v_0$
of the problem 
\[
\begin{cases}
-\dive (\C_0(x) D v_0)=f, & \text{in}\ \Omega, \\
v_0=0, & \text{on}\ \partial\Omega.
\end{cases}
\]

\end{Definition}

The main properties of $G$-convergence are the following (see, e.g., 
\cite{JKO}, Section 12.2).

\begin{teor}\label{t:gcprop}\mbox{}
\begin{itemize}
\item[(i)] The $G$-limit $\C_0$ is uniquely defined.
\item[(ii)] Let $(\C_h)$ $G$-converging to
$\C_0$ and let $w_h\in W^{1,2}(\Omega,\R^3)$ be such that $-\dive (\C_h D 
w_h)=g$ in $W^{-1,2}(\Omega,\R^3)$. 
If $w_h$ converge to $w$ weakly in  $W^{1,2}(\Omega,\R^3)$, then $\C_h D w_h$ 
converge to
$\C_0 D w$ weakly in  $L^2$.
\item[(iii)] The class $\E(c_0,c_1,\Omega)$ is compact
with respect to $G$-convergence;
\item[(iv)]  The $G$-limit $\C_0$ satisfies the two-sided estimate of Voigt-Reiss:
\[
\left(\lim_{h \to 0} (\C_h)^{-1}\right)^{-1}\leq \C_0 \leq \lim_{h \to 0} \C_h,
\]  
where the limits are understood in the sense of weak convergence in 
$L^2$. 
\end{itemize}
\end{teor}

By Theorem \ref{t:gcprop}(iii) every sequence $\C_h$ admits a $G$-converging 
subsequence. In what follows we only consider $G$-converging sequences 
$\C_h\in \E(c_0,c_1,\Omega)$ with coefficients satisfying the following uniform $VMO$ 
estimate: there exists a decreasing function $\omega\colon [0,1] \to \R$ such 
that $\lim_{r \to 0}\omega(r)=0$, and for every $r\in [0,1]$
\begin{equation}\label{f:uvmo}
\sup_{B_\rho\subseteq \Omega,\ \rho\leq r} 
\frac{1}{|B_\rho|} \int_{B_\rho} \left| \C_h(s)- \frac{1}{|B_\rho|}
\int_{B_\rho} \C_h(t)\, dt \right|\, ds \leq \omega(r), \qquad \forall h.
\end{equation}

We show that the $G$-convergence implies also the convergence of the solutions 
to elasticity problems.

\begin{teor}\label{t:gconv}
Assume that the sequence $(\C_h)\in \E(c_0,c_1,\Omega)$ satisfies the uniform $VMO$ 
condition \eqref{f:uvmo}. If $(\C_h)$
$G$-converges to  $\C_0$, then for every $\mu\in\M$ the sequence 
$\beta_h$ of solutions to

\begin{equation}\label{f:gconv}
\begin{cases}
-\dive (\C_h(x) \beta_h)=0, & \text{in}\ \Omega, \\
\curl \beta_h =\mu, & \text{in}\ \Omega, \\
\C_h\beta_h \cdot n =0, & \text{on}\ \partial\Omega \\
\end{cases}
\end{equation}

converges weakly in $\Ltm$ to the solution $\beta_0$ to

\begin{equation}\label{f:glim}
\begin{cases}
-\dive (\C_0(x) \beta_0)=0, & \text{in}\ \Omega, \\
\curl \beta_0 =\mu, & \text{in}\ \Omega, \\
\C_0\beta_0 \cdot n =0, & \text{on}\ \partial\Omega. \\
\end{cases}
\end{equation}
\end{teor}

\begin{proof}
Given $\bmu$ as in Proposition \ref{p:betamu}, let $\ue$ be the duality solution
to the problem
\[
\begin{cases}
-\dive (\C_h(x) D \ue)=\dive (\C_h \beta^\mu), & \text{in}\ \Omega, \\
\C_h D\ue \cdot n =-\C_h\beta^\mu \cdot n, & \text{on}\ \partial\Omega, 
\end{cases}. 
\]
Then $\ue \in W^{1,3/2}(\Omega,\Rtt)\cap\Rort$, 
there exists $c>0$, independent of $h$, such that
\begin{equation}\label{f:dualest}
\| \ue\|_{W^{1,3/2}(\Omega;\R^3)} \leq c \|\beta^\mu\|_{\Ltm},
\end{equation}
and 
\begin{equation}\label{f:dualhom}
\int_{\Omega}G D\ue\, dx = \int_{\Omega}\C_h\beta^\mu(x) D\ve\, dx
\qquad \forall\, G\in \X_3,
\end{equation}
where $\ve\in W^{1,3}(\Omega,\R^3)\cap \Rort$ is the solution to
\[
\begin{cases}
-\dive (\C_h(x) D \ve)=\dive G, & \text{in}\ \Omega, \\
\C_h D\ve \cdot n =-G\cdot n, & \text{on}\ \partial\Omega.
\end{cases}. 
\]
On the other hand, given $G\in \X_3$, by \eqref{f:regest} there exists $C>0$,
independent of $h$, such that
\[
\| \ve\|_{W^{1,3}(\Omega;\R^3)} \leq C \|G\|_{L^3(\Omega; \Rtt)},
\]
so that $\ve$ converges, up to a subsequence, to a function $v_0$ in the weak 
topology of $W^{1,3}(\Omega;\R^3)$.
Then, by Theorem \ref{t:gcprop}(ii) the sequence $(\C_hD v_h)$ converges to 
$\C_0D v_0$, and 
$v_0$ is the solution to the boundary value problem
\begin{equation}\label{f:limp}
\begin{cases}
-\dive (\C_0(x) D v_0)=\dive G, & \text{in}\ \Omega, \\
\C_0 D v_0\cdot n=- G \cdot n, & \text{on}\ \partial\Omega.
\end{cases}
\end{equation}
As a matter of fact, the whole sequence $v_h$ converges to $v_0$, due to the 
uniqueness of the solution of the limit problem \eqref{f:limp}.

Finally, by \eqref{f:dualest}, also the sequence $\ue$ converges, up to a 
subsequence, to a function $u_0$ in the weak topology of 
$W^{1,3/2}(\Omega;\R^3)$, and a passage to the limit in \eqref{f:dualhom}
shows that $u_0$ is the duality solution to the problem 
\[ 
\begin{cases}
-\dive (\C_0 D u_0)=\dive (\C_0\bmu), & \text{in}\ \Omega, \\
\C_0 u_0\cdot n=-\C_0\bmu \cdot n, & \text{on}\ \partial\Omega.
\end{cases}
\]
The convergence of the solutions $\beta_h$ of the problems \eqref{f:gconv} to 
the solution $\beta_0$ of the limit problem \eqref{f:glim}  now follows from 
the fact that $\beta_h=\bmu+Du_h$.
\end{proof}

As an example we finally observe that in the case of periodic rapidly 
oscillating  coefficients the effective behaviour of the corresponding 
incompatible fields is described by the homogenized effective tensor 
characterized by  the homogenization procedure of the elliptic systems in 
elasticity.

Specifically, if $Y$ denotes the reference cell $Y=[0,1]^3$, let 
$\C=\C(y)$ be a 
$Y$--periodic tensor valued function in $\E(c_0,c_1,Y)$ satisfying 
\eqref{f:VMO}, and let us consider
the asymptotic behavior as $\e \to 0$ of the solutions to 
the linear 
problems with rapidly oscillating periodic coefficients 
$\Ce(x)=\C(\frac{x}{\e})$ 
\begin{equation}\label{f:oscill}
\begin{cases}
-\dive (\Ce(x) \beta_\e)=0, & \text{in}\ \Omega, \\
\curl \beta_\e =\mu, & \text{in}\ \Omega, \\
\Ce\beta_\e\ \cdot n =0, & \text{on}\ \partial\Omega \\
\end{cases}
\end{equation}
where $\mu\in\M$.

It is well known (see, e.g. \cite{Shen}) that the sequence $(\Ce)$ is 
$G$-convergent to a effective operator with constant coefficients $\Chom$, and 
that 
$\Chom\in \E(c_0,c_1,\Omega)$. 

Moreover, under the assumption that $\C$ is $VM0$, by Theorem 4.3.1 in 
\cite{Shen}, for every $G\in \X_3$ the 
variational solution $v_\e\in W^{1,3}(\Omega,\R^3)$ to
\begin{equation}\label{f:oscillv}
\begin{cases}
-\dive (\Ce(x) Dv_\e)=\dive G, & \text{in}\ \Omega, \\
\Ce Dv_\e\ \cdot n =-G \cdot n, & \text{on}\ \partial\Omega \\
\end{cases}
\end{equation}
satisfies the estimate
\[
\|v_\e\|_{W^{1,3}(\Omega,\R^3)} \leq C \|G\|_{L^3(\Rtt)}
\]
with a constant $C>0$ independent of $\e$. Hence, following the lines of the 
proof of Theorem \ref{p:dual}, we obtain that the duality solutions to the 
problems 
\[
\begin{cases}
-\dive (\Ce(x) D u_e)=\dive F, & \text{in}\ \Omega, \\
\Ce Du_e \cdot n =-F \cdot n, & \text{on}\ \partial\Omega, 
\end{cases}. 
\]
satisfy the estimate
\[
\| u_e\|_{W^{1,3/2}(\Omega;\R^3)} \leq c \|F|_{\Ltm},
\]
with a constant $C>0$ independent of $\e$.

In conclusion, following the lines of Theorem \ref{t:gconv}
we obtain that, if $\C=\C(y)$ is a 
$Y$--periodic tensor valued function in $\E(c_0,c_1,Y)$ satisfying 
\eqref{f:VMO}, then for every $\mu\in\M$, the solutions $\beta_\e$ to 
\eqref{f:oscill} converge weakly in $\Ltm$ to the 
solution $\beta$ to the problem
\[
\begin{cases}
-\dive (\Chom \beta)=0, & \text{in}\ \Omega, \\
\curl \beta =\mu, & \text{in}\ \Omega, \\
\Chom\beta \cdot n =0, & \text{on}\ \partial\Omega. \\
\end{cases}
\] 


\end{document}